\documentclass{amsart}

\usepackage[dvipdfmx]{graphicx}
\usepackage{color}

\usepackage{float}
\usepackage{array,booktabs}
\usepackage{amssymb}
\usepackage{subcaption}
\usepackage{txfonts}
\usepackage{amsmath}
\usepackage{subcaption}
\usepackage{bm}
\usepackage{mathrsfs}
\usepackage{overpic}
\usepackage{multimedia}

\newtheorem{theorem}{Theorem}[section]
\newtheorem{lemma}[theorem]{Lemma}

\theoremstyle{definition}
\newtheorem{definition}[theorem]{Definition}
\newtheorem{example}[theorem]{Example}

\newtheorem{cor}[theorem]{Corollary}
\theoremstyle{remark}

\newcommand{\ack}{\noindent \textbf{Acknowledgements. }}
\newcommand{\org}{\noindent \textbf{Organization. }}

\DeclareMathOperator{\interior}{int}
\DeclareMathOperator{\aff}{Aff}
\DeclareMathOperator{\closure}{cl}
\DeclareMathOperator{\bip}{Bip}
\DeclareMathOperator{\ehr}{Ehr}
\DeclareMathOperator{\T}{Top}

\newcommand{\N}{\mathbb{N}}
\newcommand{\R}{\mathbb{R}}
\newcommand{\conv}{\operatorname{Conv}}

\numberwithin{equation}{section}

\title{A mirroring formula for the interior polynomial of a bipartite graph}
\author{Keiju Kato}
\date{\today}
\email{kato.k.at@m.titech.ac.jp}
\address{Department of Mathematics, Tokyo Institute of Technology, Oh-okayama 2-12-1, Meguro-ku, Tokyo 152-8551, Japan}

\begin{document}
\maketitle

\begin{abstract}
The interior polynomial is an invariant of (signed) bipartite graphs, and the interior polynomial of a plane bipartite graph is equal to a part of the HOMFLY polynomial of a naturally associated link. The HOMFLY polynomial $P_L(v,z)$ is a famous link invariant with many known properties. For example, the HOMFLY polynomial of the mirror image of $L$ is given by $P_{L}(-v^{-1},z)$. This implies a property of the interior polynomial in the planar case. We prove that the same property holds for any bipartite graph. The proof relies on Ehrhart reciprocity applied to the so called root polytope. We also establish formulas for the interior polynomial inspired by the knot theoretical notions of flyping and mutation.
\end{abstract}

\section{Introduction}
In this paper, we investigate properties of the interior polynomial, which is an invariant of bipartite graphs. A priori, the interior polynomial is an invariant of hypergraphs defined by K\'alm\'an \cite{K}. Here a hypergraph $\mathscr{H}=(V,E)$ has a vertex set $V$ and a hyperedge set $E$, where $E$ is a multiset of non-empty subsets of $V$. The interior polynomial is naturally associated to this structure, but by the main result of \cite{KP}, we may also regard the interior polynomial as an invariant of the natural bipartite graph $\bip \mathscr{H}$ with color classes $E$ and $V$. The author extended the interior polynomial to signed bipartite graphs, that is, bipartite graphs $G$ with a sign $\mathcal{E} \to \{ +1 , -1 \}$, where $\mathcal{E}$ is the set of edges \cite{kato}. The signed interior polynomial $I^+_G$ is constructed as an alternating sum of the interior polynomials of the bipartite graphs obtained from $G$ by deleting some negative edges and forgetting the sign.

The interior polynomial is related to a part of the HOMFLY polynomial. The HOMFLY polynomial \cite{homfly} is a two-variable invariant of oriented links in $S^3$ defined by the skein relation
\[
v^{-1} P_{\includegraphics[width=0.35cm]{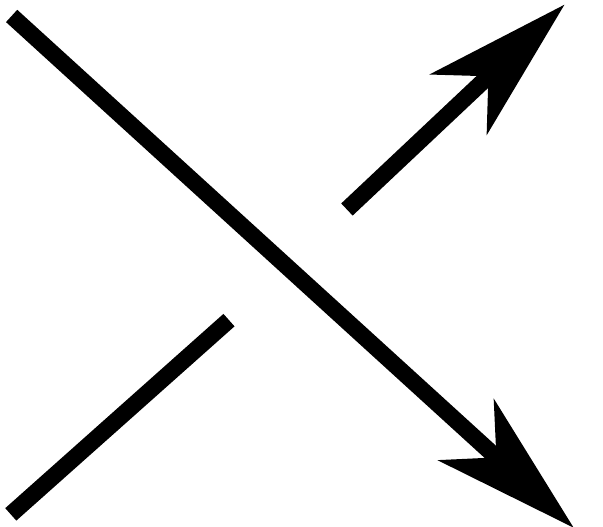}}(v,z)-
v      P_{\includegraphics[width=0.35cm]{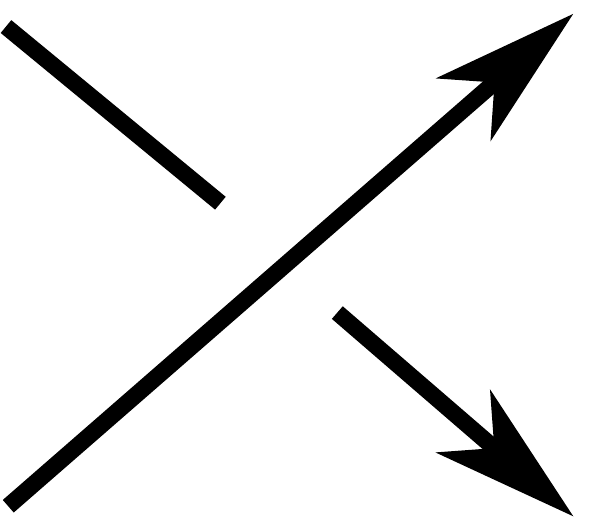}}(v,z)=
z      P_{\includegraphics[width=0.35cm]{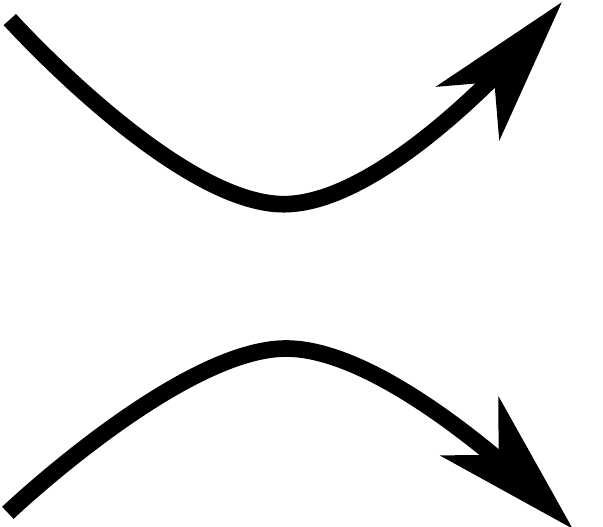}}(v,z),
\]
and the initial condition $P_{\mbox{\footnotesize unknot}}(v,z)=1$. Regarding arbitrary links, Morton \cite{morton} showed that the maximal $z$ exponent in the HOMFLY polynomial of an oriented link diagram $D$ is less than or equal to $c(D)-s(D)+1$, where $c(D)$ is the crossing number of $D$ and $s(D)$ is the number of its Seifert circles. We call the coefficient of $z^{c(D)-s(D)+1}$, which is a polynomial in $v$, the top of the HOMFLY polynomial and denote it by $\T_D(v)$. When $G=(V,E,\mathcal{E}_+\cup\mathcal{E}_-)$ is a signed plane bipartite graph, the coefficients of the interior polynomial $I^+_G(x)$ agree with the coefficients of $\T_{L_G}(v)$, where $L_G$ is the special link diagram with Seifert graph $G$ \cite{kato}. More precisely,
\[
\T_{L_G}(v)=v^{|\mathcal{E}_{+}|-|\mathcal{E}_{-}|-(|E|+|V|)+1}I^{+}_{G}(v^2 ).
\]

This correspondence follows from its special case when $G$ is a positive graph and that in turn is established in two steps. First, the interior polynomial of $G$ is equivalent to the Ehrhart polynomial of the root polytope of $G$ \cite{KP}. The latter can be thought of as an h-vector \cite{KP} and coincides with $\T_{L_G}$ \cite{KM}.

The HOMFLY polynomial has many properties. For example, for any link $L$, the HOMFLY polynomial of the mirror image $L^*$ satisfies $P_{L^*}(v,z)=P_{L}(-v^{-1},z)$. So the coefficients of $\T_{L^*}(v)$ are obtained from those of $\T_{L}(v)$ by reversing the order (and possibly an overall sign change, which occurs exactly when $L$ has an even number of components). If $G$ is the Seifert graph of $L$, the Seifert graph of $L^*$ is obtained from $G$ by changing all the signs. We denote that graph by $-G$. The main result of the paper is that the connection between $I_G^+$ and $I_{-G}^+$, provided by the HOMFLY polynomial in the planar case, is true in general. Namely, we have the following.
\begin{theorem}\label{thm:signchange}
For any signed bipartite graph $G=(V,E,\mathcal{E}_+\cup\mathcal{E}_-)$, let $-G$ be the signed bipartite graph obtained from $G$ by changing all the signs. Then
\[
(-1)^{|\mathcal{E}_+|+|\mathcal{E}_-|+|E|+|V|-1}x^{|E|+|V|-1}I^+_G(1/x)=I^+_{-G}(x).
\]
\end{theorem}
The proof is based on Ehrhart reciprocity of the root polytope and the following result on interiors of convex hulls, which may be interesting in its own right.
\begin{theorem}\label{conv}
For any finite set $X=\{x_0,\ldots,x_{n}\}\subset\mathbb{R}^d$, we have
\begin{equation}\label{eq:ind}
(-1)^{\dim(\conv X)}[\interior{(\conv X)}]=\sum_{S\subseteq X}(-1)^{|S|-1}[\conv S],
\end{equation}
where $\conv$ means convex hull, $\interior$ means relative interior, and $[A]$ stands for the indicator function of $A\subseteq\mathbb{R}^d$.
\end{theorem}
 
Flyping and mutation are link operations which do not change the HOMFLY polynomial. The link obtained by flyping is, in fact, ambient isotopic to the original. Mutation may change the link type but it leaves the HOMFLY polynomial invariant. Based on the Seifert graphs of the link diagrams before and after flyping or mutation, we define flyping and mutation for signed bipartite graphs and we obtain the following theorem, which follows relatively easily from the recursion relation established in \cite{kato}.
\begin{theorem}\label{thm:nochange}
For any signed bipartite graph, the interior polynomial does not change under graph flyping and graph mutation.
\end{theorem}

\org In section \ref{sec:int}, we recall some definitions and facts about the interior polynomial for signed bipartite graphs and the Ehrhart polynomial of the root polytope. In section \ref{sec:subgraph}, we prove Theorem \ref{conv} and the unsigned version of Theorem \ref{thm:signchange}. In section \ref{sec:mirror}, we prove Theorem \ref{thm:signchange}, the main theorem in this paper. In section \ref{sec:fm}, we prove Theorem \ref{thm:nochange}.

\ack I should like to express my gratitude to associate professor Tam\'as K\'alm\'an for constant encouragement and much helpful advice.

\section{Preliminaries}\label{sec:int}
\subsection{Interior polynomial and HOMFLY polynomial}
A hypergraph is a pair $\mathscr{H}=(V,E)$, where $V$ is a finite set and $E$ is a finite multiset of non-empty subsets of $V$. We order the set $E$ of hyperedges and define the interior polynomial of $\mathscr{H}$ using the activity relation between hyperedges and so called hypertrees \cite{K}. For the set of hypertrees to be non-empty, here we assume that $\mathscr{H}$ is connected. This means that the graph $\bip\mathscr{H}$, defined below, is connected. The interior polynomial does not depend on the order of the hyperedges. It generalizes the evaluation $x^{|V|-1}T_{G}(1/x,1)$ of the classical Tutte polynomial $T_G(x,y)$ of the graph $G=(V,E)$. We will not review its definition in detail here because it will suffice to rely on a recursive property, stated in Theorem \ref{recursion} below.

We obtain a bipartite graph from the hypergraph $\mathscr{H}=(V,E)$, by letting an edge of the bipartite graph connect a vertex (i.e., an element of $V$) and a hyperedge if the hyperedge contains the vertex. We denote the bipartite graph obtained from the hypergraph $\mathscr{H}$ by $\bip\mathscr{H}=(V,E,\mathcal{E})$. Thus $V$ and $E$ become the color classes of $\bip\mathscr{H}$; in particular, both play the role of vertices. This construction gives a two-to-one correspondence from hypergraphs to bipartite graphs. The two hypergraphs corresponding to the same bipartite graph are called abstract dual. We will denote by $\overline{\mathscr{H}}=(E,V)$ the abstract dual hypergraph of $\mathscr{H}=(V,E)$. Whenever one connected bipartite graph generates two hypergraphs in this way, the interior polynomials of them are the same \cite{KP}. Therefore we may regard the interior polynomial as an invariant of bipartite graphs.

When the bipartite graph $G$ has $k(G)$ components, letting $G=G_1\cup G_2 \cup \cdots \cup G_{k(G)}$, the interior polynomial of $G$ is defined by $I'_{G}\left(x \right)=\left(1-x \right)^{k(G)-1}\prod_{i=1}^{k(G)} I_{G_i}\left(x \right)$. Next we  define the interior polynomial for a signed bipartite graph. Let $G=(V,E,\mathcal{E}_+\cup\mathcal{E}_-)$ be a signed bipartite graph, where $\mathcal{E}_+$ is the positive edge set and $\mathcal{E}_-$ is the negative edge set.
Let $\mathcal{S}$ be a subset of $\mathcal{E}_-$. The unsigned bipartite graph $G \setminus \mathcal{S}$ is obtained from G by deleting all edges in $\mathcal{S}$ and forgetting the signs of the remaining edges. So we may compute the interior polynomial of $G \setminus \mathcal{S}$. We will construct the interior polynomial of a signed bipartite graph as follows.

\begin{definition}
Let $G=(V,E,\mathcal{E}_+\cup\mathcal{E}_-)$ be a signed bipartite graph. We define the signed interior polynomial as
\[
I^{+}_{G}\left(x \right)=\sum_{\mathcal{S} \subseteq \mathcal{E}_{-}}(-1)^{|\mathcal{S}|}I'_{G \setminus \mathcal{S}}(x).
\]
\end{definition}

The abstract theory outlined above may be applied in knot theory as follows. Let $L_{G}$ be the special alternating diagram obtained from the unsigned plane bipartite graph $G=(V,E,\mathcal{E})$ by replacing each edge by a positive crossing. This is known as median construction; see Figure \ref{fig:52} for an example.
\begin{theorem}[T.\ K\'alm\'an, H.\ Murakami and A.\ Postnikov, \cite{KM,KP}]\label{thm:KMP}
For any plane connected bipartite graph $G=(V,E,\mathcal{E})$, we have
\[
\T_{L_G}(v)=v^{|\mathcal{E}|-(|V|+|E|)+1}I_G(v^2).
\]
\end{theorem}

For any signed bipartite graph $G$, the link diagram $L_G$ is obtained from $G$ by replacing edges with positive and negative crossings, as shown in Figure \ref{fig:cro}. The author extended Theorem \ref{thm:KMP} to signed bipartite graphs.

\begin{figure}[htbp]
\begin{tabular}{ccc}
\begin{minipage}{0.3\hsize}
\centering
\includegraphics[width=1.5cm]{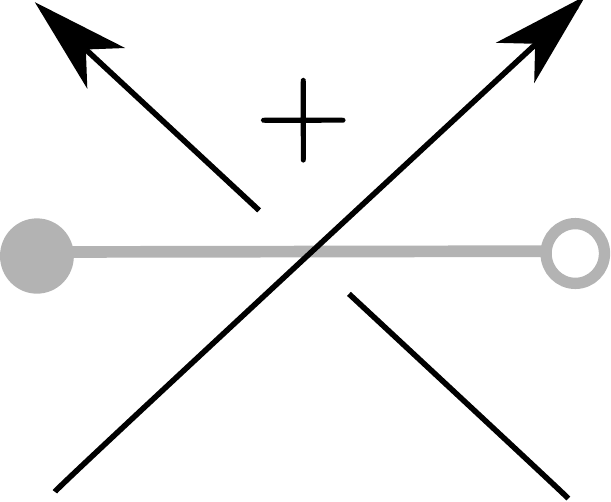} 
\end{minipage}
\begin{minipage}{0.3\hsize}
\centering
\includegraphics[width=1.5cm]{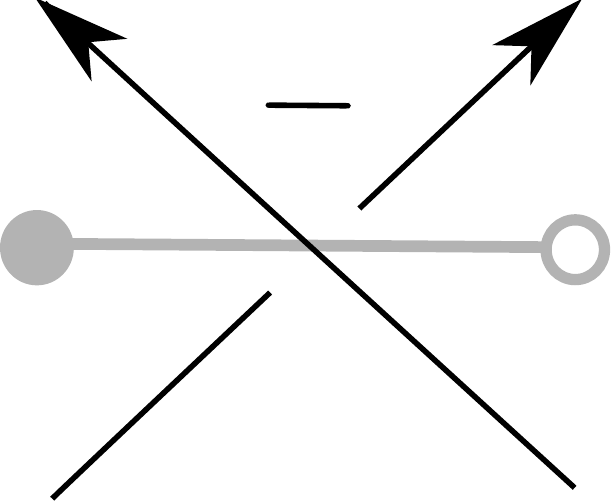}
\end{minipage}
\end{tabular}
\caption{The positive crossing and the negative crossing corresponding to a positive edge and a negative edge, respectively.}\label{fig:cro}
\end{figure}

\begin{theorem}[\cite{kato}]\label{thm:sign}
Let $G=(V,E,\mathcal{E}_+ \cup \mathcal{E}_-)$ be a plane signed bipartite graph. Then we have
\[
\T_{L_G}(v)=v^{|\mathcal{E}_{+}|-|\mathcal{E}_{-}|-(|V|+|E|)+1}I^{+}_{G}(v^2 ).
\]
\end{theorem}

Applying this theorem, for any plane signed bipartite graph, we get properties of the interior polynomial from properties of the HOMFLY polynomial. In sections \ref{sec:mirror} and \ref{sec:fm}, we will extend some of these properties to all signed bipartite graphs.

We recall two properties of the interior polynomial that we need for the proof of our main theorem.
\begin{theorem}[\cite{kato}]\label{recursion}
If an unsigned bipartite graph $G$ contains a cycle $\epsilon_1,\delta_1,\epsilon_2,\delta_2,\cdots,\epsilon_n,\delta_n$, then we have
\[
I'_G(x)=\sum_{\emptyset\neq\mathcal{S}\subset\{\epsilon_1,\epsilon_2,\ldots,\epsilon_n\}}(-1)^{|\mathcal{S}|-1}I'_{G\setminus\mathcal{S}}(x).
\]
\end{theorem}
This is one possible counterpart of the deletion-contraction relation of the Tutte polynomial, in that it enables one to compute the interior polynomial recursively.

Another property of the interior polynomial is related to the skein relation of the HOMFLY polynomial. Let $G$ be a signed bipartite graph and let $\epsilon$ be one of the negative edges in $G$. The bipartite graph $G \setminus \epsilon$ is obtained from $G$ by deleting $\epsilon$ and $G+\epsilon$ is obtained from $G$ by replacing the negative edge $\epsilon$ by a positive edge (see Figure \ref{fig:replace}).

\begin{figure}[htbp]
\begin{tabular}{ccc}
\begin{minipage}{0.33\hsize}
\begin{center}
\includegraphics[width=3.5cm]{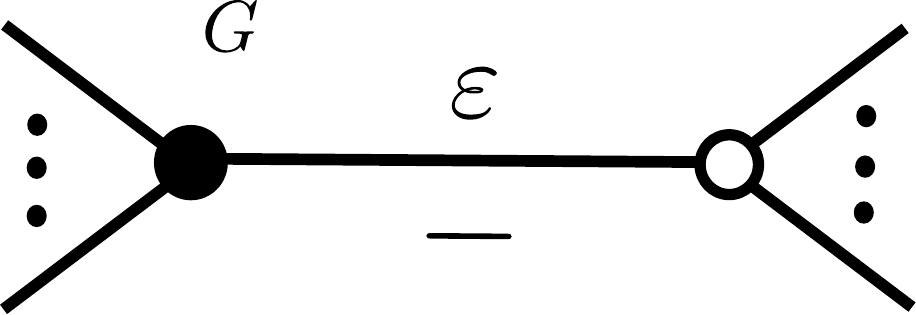}
\end{center}
\end{minipage}
\begin{minipage}{0.33\hsize}
\begin{center}
\includegraphics[width=3.5cm]{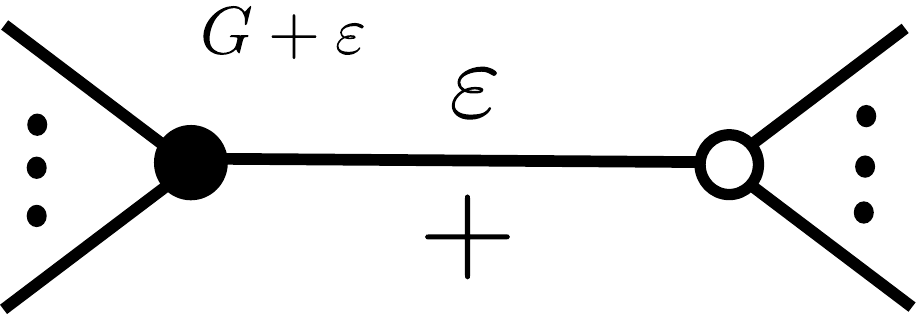}
\end{center}
\end{minipage}
\begin{minipage}{0.33\hsize}
\begin{center}
\includegraphics[width=3.5cm]{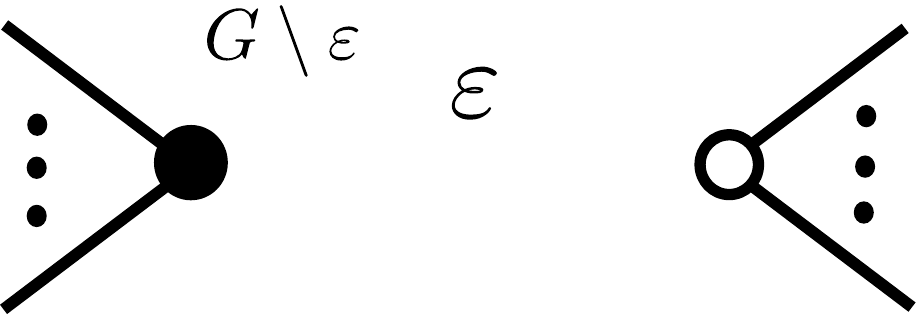}
\end{center}
\end{minipage}
\end{tabular}
\caption{A version of the skein triple for signed bipartite graph.}\label{fig:replace}
\end{figure}

\begin{lemma}[\cite{kato}]\label{lem:sumint}
Let $G$ be a signed bipartite graph and let $\epsilon$ be one of the negative edges in $G$. Then we have $I^+_G(x)=I^+_{G+\epsilon}(x)-I^+_{G \setminus \epsilon}(x)$.
\end{lemma}
This lemma will be needed in the part of the proof using induction on the number of negative edges.

\subsection{Ehrhart polynomial}
In \cite{KP}, the interior polynomial of an unsigned bipartite graph $G$ is shown to be equivalent to the Ehrhart polynomial of the root polytope of $G$. We review some details. The root polytope of a bipartite graph is defined as follows.
\begin{definition}
Let $G=(V,E,\mathcal{E})$ be a bipartite graph. For $e \in E$ and $v \in V$, let {\bf e} and {\bf v} denote the corresponding standard generators of $\mathbb{R}^E \oplus \mathbb{R}^V$. Define the root polytope of $G$ by
\[
Q_G=\conv \{\, {\bf e} + {\bf v} \,|\, ev \mbox{ is an edge of }G \,\}.
\]
\end{definition}
We know that when $G$ is connected, $\dim Q_G = |V|+|E|-2$ \cite{P}, and let $d=|V|+|E|-2$ in this paper.
\begin{definition}
Let $G=(V,E,\mathcal{E})$ be a bipartite graph and $Q_G$ be the root polytope of $G$. For any positive integer $s$, the Ehrhart polynomial is defined by
\[
\varepsilon_{Q_G}(s) = |(s\cdot Q_G) \cap (\mathbb{Z}^{E}\oplus\mathbb{Z}^{V})|.
\]
\end{definition}
In general, for any polytope $P$, the analogously defined $\varepsilon_{P}(s)$ is not a polynomial. However, for a convex polytope $P$ whose vertices are integer points, $\varepsilon_{P}(s)$ is a polynomial. Thus, $\varepsilon_{Q_G}(s)$ is a polynomial.
\begin{definition}
Let $G$ be a bipartite graph and $\varepsilon_{Q_G}(s)$ be the Ehrhart polynomial of the root polytope $Q_G$. The Ehrhart series is defined by
\[
\ehr_{Q_G}(x)=1+\sum_{s\in \mathbb{N}}\varepsilon_{Q_G}(s)x^s.
\]
\end{definition}
Notice that, for a (bipartite) graph with no edges, we have $\ehr_{Q_G}(x)=1$. Now the Ehrhart series of the root polytope $Q_G$ is equivalent to the interior polynomial of the bipartite graph $G$. 
\begin{theorem}\label{lem:serint}
Let $G=(V,E,\mathcal{E})$ be a connected bipartite graph and $I_G(x)$ be the interior polynomial of $G$. Then
\[
\frac{I_G(x)}{(1-x)^{|E|+|V|-1}}=\ehr_{Q_G}(x).
\]
\end{theorem}
This theorem is implicit in \cite{KP}. The author extended this theorem to any (unsigned but possibly disconnected) bipartite graph.
\begin{theorem}[\cite{kato}]\label{thm:disserint}
Let $G=(V,E,\mathcal{E})$ be a bipartite graph and $I'_G(x)$ be the interior polynomial of $G$. Then
\[
\frac{I'_G(x)}{(1-x)^{|E|+|V|-1}}=\ehr_{Q_G}(x).
\]
\end{theorem}
\begin{definition}
Let $G=(V,E,\mathcal{E}_+\cup\mathcal{E}_-)$ be a signed bipartite graph.
We define the signed Ehrhart series as
\[
\ehr^+_{G}(x)=\sum_{\mathcal{S} \subseteq \mathcal{E}_{-}}(-1)^{|\mathcal{S}|}\ehr_{Q_{G \setminus \mathcal{S}}}\left(x \right),
\]
where the graph $G\setminus\mathcal{S}$ is treated as unsigned.
\end{definition}
Now the signed interior polynomial is equivalent to the signed Ehrhart series.
\begin{theorem}[\cite{kato}]\label{thm:signserint}
Let $G=(V,E,\mathcal{E})$ be a signed bipartite graph and $I^+_G(x)$ be the signed interior polynomial of $G$. Then
\[
\frac{I^+_G(x)}{(1-x)^{|E|+|V|-1}}=\ehr^+_{G}(x).
\]
\end{theorem}

\section{Subgraph expansion of the interior polynomial}\label{sec:subgraph}
Before we show Theorem \ref{thm:signchange}, we prove the following property of the interior polynomial for unsigned bipartite graphs. 
\begin{theorem}\label{thm:subgraph}
Let $G=(V,E,\mathcal{E})$ be a bipartite graph. For any edge set $\mathcal{S}\subseteq \mathcal{E}$, we may consider the subgraph $\mathcal{S}=(V,E,\mathcal{S})$. Then,
\[
(-x)^{|E|+|V|-1}I'_G(1/x)=
\sum_{\mathcal{S}\subseteq\mathcal{E}}(-1)^{|\mathcal{S}|}I'_{\mathcal{S}}(x).
\]
\end{theorem}
This theorem is equivalent to Theorem \ref{thm:signchange} when the signed bipartite graph $G$ has only positive edges. To prove Theorem \ref{thm:subgraph}, we need two other theorems. The first is a well known formula, called Ehrhart reciprocity \cite[Theorem 4.4]{CCD}. Here a rational convex polytope is a convex polytope whose vertices have only rational coordinates. The root polytope of any bipartite graph is a rational convex polytope.
\begin{theorem}[Ehrhart reciprocity]\label{ER}
Let $P$ be a rational convex polytope. Then,
\[
\ehr_P(1/x)=(-1)^{\dim{P}+1}\ehr_{\interior{P}}(x).
\]
\end{theorem}
The second theorem is Theorem \ref{conv}. For any set $A \subset \mathbb{R}^{n}$, let the function $[A]:\mathbb{R}^{n} \to \mathbb{R}$ be defined by 
\begin{equation*}
[A](x) = \begin{cases}
           1 & (x \in A) \\
           0 & (x \notin A)
         \end{cases}.
\end{equation*}
We call this function the indicator function of $A$. Now we will prove Theorem \ref{conv}.

\begin{proof}[Proof of Theorem \ref{conv}]
When $x\notin \conv X$, it is clear that $[\conv S](x)=0$ for any subset $S\subset X$. Therefore, both sides of \ref{eq:ind} are equal to $0$.

Now, for $x \in \conv X$, we have to show
\begin{eqnarray*}
\sum_{S\subseteq X}(-1)^{|S|-1}[\conv S](x)
=\begin{cases}
        (-1)^{\dim(\conv X)}& (x \in \interior(\conv X)) \\
         0                  & (x \in \partial (\conv X))
 \end{cases},
\end{eqnarray*}
where $\partial A = \closure(A)\setminus\interior A$ stands for relative boundary.

First we do this in the boudary case $x\in \partial(\conv X)$. Let $X'\subset X$ be the set of all elements of $X$ along the minimal face of $\conv X$ containing $x$. As $x\in \partial(\conv X)$, we have $X'\ne X$. For all $S\subseteq X$ such that $x\in \conv S$, we also have $x\in\conv(S\cap X')$. For any $S'\subseteq X'$ such that $x\in \conv S'$, there is even number of $S\subseteq X$ such that $S'=S\cap X'$ and the alternating sum of their indicators at $x$ is $0$. This completes the proof in the case $x\in\partial(\conv X)$.

It remains to show (\ref{eq:ind}) in case $x \in \interior(\conv X)$. When $X$ is affine independent, the claim is obvious (note that $\dim(\conv X)=|X|-1$). When $x_0,\ldots,x_{n}$ are affine dependent, we apply induction through decreasing dimension. (Formally, the induction is on $|X|-\dim(\conv X)$. The affine independent case is when this value is $1$. The value of $|X|$ will stay fixed.) Let $y_i=(x_i, \epsilon_i)$ for $i=0,\ldots,n$ and $Y=\{y_0,\ldots,y_{n}\}$. We will denote by $p:\R^{d+1}\rightarrow \R^d$ the projection such that $p(y_i)=x_i$ for $i=0,\ldots,n$. There exists $\{\epsilon_0,\ldots,\epsilon_{n}\}$ such that $\dim(\conv Y)=\dim(\conv X)+1$. We assume that, for $y\in\conv Y$, we have
\begin{eqnarray*}
\sum_{S\subseteq Y}(-1)^{|S|-1}[\conv S](y)
&=&(-1)^{\dim(\conv Y)}[\interior{(\conv Y)}](y)\\
&=&\begin{cases}
        (-1)^{\dim(\conv Y)}& (y \in \interior(\conv Y)) \\
         0                  & (y \in \partial (\conv Y))
   \end{cases}.
\end{eqnarray*}
The set $p^{-1}(x)\cap\conv Y$ is a segment of positive length and it intersects some hyperplanes $\aff {S'}$ (for suitable $S'\subseteq Y$) at the points $a_j$ ($1\le j\le k$). We take the indices of the $a_j$ to satisfy $j_1< j_2$ if and only if $\epsilon_{j_1}<\epsilon_{j_2}$, where $\epsilon_{j}$ is the last coordinate of $a_j$. Let $A=\{a_1,\ldots,a_k\}$. We define $B=\{b_1,\ldots,b_{k-1}\}$, where $b_j=(a_j+a_{j+1})/2$ ($1\le j \le k-1$). For any $S'\subseteq Y$, the set $\conv S' \cap p^{-1}(x)$ is either empty, a point, or a segment of $\R ^{d+1}$. When $\conv S' \cap p^{-1}(x)$ is a point, we have $|\{a_j|a_j\in \conv S'\}|=1,|\{b_j|b_j\in \conv S'\}|=0$. When $\conv S' \cap p^{-1}(x)$ is a segment, we have $|\{a_j|a_j\in \conv S'\}|=|\{b_j|b_j\in \conv S'\}|+1$. In both cases, $|\{a_j|a_j\in \conv S'\}|-|\{b_j|b_j\in \conv S'\}|=1$.
So, we have
\begin{eqnarray*}
\sum_{S\subseteq X}(-1)^{|S|-1}[\conv S](x)
&=&\sum_{S \subseteq X \atop x\in \conv S}(-1)^{|S|-1}\\
&=&\sum_{S' \subseteq Y \atop \conv S' \cap p^{-1}(x)\ne\emptyset}(-1)^{|S'|-1}\\
&=&\sum_{S' \subseteq Y}(-1)^{|S'|-1}\left(|\{a_j|a_j\in \conv S'\}|-|\{b_j|b_j\in \conv S'\}|\right)\\
&=&\sum_{S' \subseteq Y}(-1)^{|S'|-1}\left(\sum_{j=1}^{k}[\conv S'](a_j)-\sum_{j=1}^{k-1}[\conv S'](b_j)\right)\\
&=&\sum_{j=1}^{k}\sum_{S' \subseteq Y}(-1)^{|S'|-1}[\conv S'](a_j)-\sum_{j=1}^{k-1}\sum_{S' \subseteq Y}(-1)^{|S'|-1}[\conv S'](b_j).
\end{eqnarray*}
By using the inductive hypothesis, and the already established boundary case for $a_1$ and $a_k$, we obtain
\begin{eqnarray*}
\sum_{S\subseteq X}(-1)^{|S|-1}[\conv S](x)
&=&\sum_{j=1}^{k}(-1)^{\dim \conv Y}[\interior \conv Y](a_j)-\sum_{j=1}^{k-1}(-1)^{\dim\conv Y}[\interior \conv Y](b_j)\\
&=&(k-2)(-1)^{\dim \conv Y}-(k-1)(-1)^{\dim\conv Y}\\
&=&(-1)^{\dim \conv Y -1}\\
&=&(-1)^{\dim \conv X}.
\end{eqnarray*}
This completes the proof by induction.
\end{proof}
Now we are in a position to prove Theorem \ref{thm:subgraph}.

\begin{proof}[Proof of Theorem \ref{thm:subgraph}]
We start with the connected case. First, we apply Theorem \ref{ER} to the root polytope $Q_G$ of the bipartite graph $G$ to obtain
\begin{equation}\label{eq:RR}
\ehr_{Q_G}(1/x)=(-1)^{d+1}\ehr_{\interior{Q_G}}(x).
\end{equation}
For any $\mathcal{S}\subseteq\mathcal{E}$, we will denote by $Q_\mathcal{S}$ the root polytope of the bipartite graph $(V,E,\mathcal{S})$. Since the root polytopes $Q_G,Q_{\mathcal{S}}$ ($\mathcal{S}\subseteq\mathcal{E}$) are convex hulls, Theorem \ref{conv} implies
\[
(-1)^{d}[\interior{Q_G}]=\sum_{\mathcal{S}\subseteq \mathcal{E}}(-1)^{|\mathcal{S}|-1}[Q_{\mathcal{S}}].
\]
By the definition of the Ehrhart polynomial, for any $s\in\N$, we have
\[
(-1)^{d}\varepsilon_{\interior{Q_G}}(s)
=\sum_{\mathcal{S}\subseteq \mathcal{E}}(-1)^{|\mathcal{S}|-1}\varepsilon_{Q_{\mathcal{S}}}(s).
\]
By the definition of the Ehrhart series, we have
\begin{equation}\label{eq:IS}
(-1)^{d}\ehr_{\interior{Q_G}}(x)
=\sum_{\mathcal{S}\subseteq \mathcal{E}}(-1)^{|\mathcal{S}|-1}\ehr_{Q_\mathcal{S}}(x).
\end{equation}
From equations (\ref{eq:RR}) and (\ref{eq:IS}), we obtain
\[
\ehr_{Q_G}(1/x)=\sum_{\mathcal{S}\subseteq \mathcal{E}}(-1)^{|\mathcal{S}|}\ehr_{Q_\mathcal{S}}(x).
\]
Now Theorem \ref{thm:disserint} yields 
\[
\frac{I'_G(1/x)}{(1-1/x)^{d+1}}
=\sum_{\mathcal{S}\subseteq\mathcal{E}}(-1)^{|\mathcal{S}|}\frac{I'_{\mathcal{S}}(x)}{(1-x)^{d+1}},
\]
from which we get
\[
(-x)^{d+1}I'_G(1/x)=\sum_{\mathcal{S}\subseteq\mathcal{E}}(-1)^{|\mathcal{S}|}I'_{\mathcal{S}}(x),
\]
which completes the proof. For a disconnected bipartite graph, the claim follow from the connected case and the definition of $I'$. 
\end{proof}
\section{Mirroring formula}\label{sec:mirror}
The goal of this section is to extend Theorem \ref{thm:subgraph} to Theorem \ref{thm:signchange}. First we recall the motivation behind both of these statements. For any link, the mirror image is obtained by reflecting it in a plane. When the link is given by a diagram $D$, we may take the plane to be projection plane and write $D^*$ for the mirror image. The following property of the HOMFLY polynomial is well known.
\begin{theorem}\label{thm:mirror}
Let $D^*$ be the mirror image of the link diagram $D$. Then,
\[
P_{D^*}(v,z)=P_D(-v^{-1},z).
\]
\end{theorem}
\begin{example}
In Figure \ref{fig:52}, the left link diagram is $5_2$ in Rolfsen's table, and the right link diagram is its mirror image. We compute the HOMFLY polynomials of these as follows. We check the formula $P_{D^*}(v,z)=P_D(-v^{-1},z)$.

\[\hspace{30pt}
\begin{array}{llllllllll}
 &\hspace{-40pt}P_{5_2}(v,z)&         &     &    & & P_{{5_2}^*}(v,z)&           &\\
=&+1v^2z^2    &+1v^4z^2 &     &    &&\hspace{20pt}=    &           &+1v^{-4}z^2  &+1v^{-2}z^2 \\
 &+1v^2z^0    &+1v^4z^0 &-1v^6z^0.& & &&\hspace{-20pt}-1v^{-6}z^{0}  &+1v^{-4}z^{0}&+1v^{-2}z^{0}.
\end{array}
\]
\begin{figure}[htbp]
\begin{tabular}{ccc}
\begin{minipage}{0.5\hsize}
\centering
\includegraphics[width=3cm]{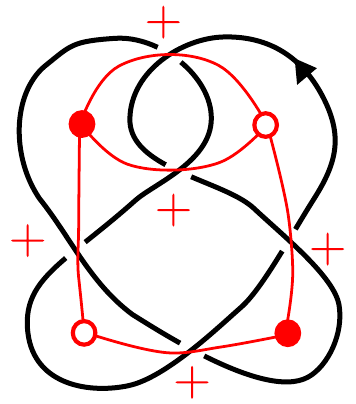}
\end{minipage}
\begin{minipage}{0.5\hsize}
\centering
\includegraphics[width=3cm]{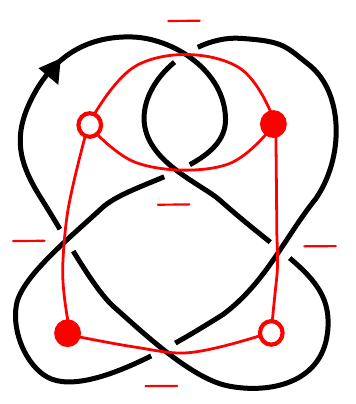}
\end{minipage}
\end{tabular}
\caption{The left diagram represents the knot $5_2$ and the right diagram is its mirror image. The bipartite graphs are their Seifert graphs.}\label{fig:52}
\end{figure}
\end{example}

In the mirror image of a link diagram, positive crossings change to negative and negative crossings change to positive. The Seifert graphs are equivalent but all signs are opposite. The coefficients of $\T_{L^*}(v)$ are obtained from those of $\T_{L}(v)$ by reversing the order. Hence the same is true for the signed interior polynomials $I^+$ of the two Seifert graphs. This phenomenon is true in the planar case, and we expect it to be true in general. For any signed bipartite graph $G$, the signed bipartite graph $-G$ is obtained from $G$ by changing all the signs. Next we prove our main theorem.

\begin{proof}[Proof of Theorem \ref{thm:signchange}]
We procced by induction on the number of negative edges $|\mathcal{E}_-|$. When $|\mathcal{E}_-|=0$, the edge set of $G$ is $\mathcal{E}=\mathcal{E}_+$. We apply Theorem \ref{thm:subgraph} to the bipartite graph $G$, forgetting sign to obtain
\[
(-x)^{d+1}I'_G(1/x)=\sum_{\mathcal{S}\subseteq\mathcal{E}}(-1)^{|\mathcal{S}|}I'_{\mathcal{S}}(x),
\]
where $d=|E|+|V|-2$. By the definition of the signed interior polynomial, we have $I^+_G(x)=I'_G(x)$ and 
\[
I^+_{-G}(x)=\sum_{\mathcal{S}\subseteq\mathcal{E}}(-1)^{|\mathcal{E}|-|\mathcal{S}|}I'_{\mathcal{S}}(x),
\]
where $\mathcal{S}$ is the bipartite graph $(V,E,\mathcal{S})$, forgetting sign. By the above, we get
\[
(-1)^{|\mathcal{E}|+d+1}x^{d+1}I^+_G(1/x)=I^+_{-G}(x).
\]
Therefore the statement holds when $|\mathcal{E}_-|=0$.

When the bipartite graph has $k$ negative edges, we suppose that the theorem holds. Let a bipartite graph $G=(V,E,\mathcal{E}_+\cup\mathcal{E}_-)$ have $k+1$ negative edges. We take a negative edge $\epsilon$ in $G$. From Lemma \ref{lem:sumint}, we have $I^+_G(x)=I^+_{G+\epsilon}(x)-I^+_{G \setminus \epsilon}(x)$, respectively. The number of the positive and negative edges of $G+\epsilon$ is $|\mathcal{E}_+|+1$ and $|\mathcal{E}_-|-1=k$, respectively. The number of the positive and negative edges of $G\setminus\epsilon$ is $|\mathcal{E}_+|$ and $|\mathcal{E}_-|-1=k$, respectively. Now by the inductive hypothesis applied to $G+\epsilon$ and $G\setminus\epsilon$, we have
\begin{eqnarray*}
I^+_G(1/x)
&=&I^+_{G+\epsilon}(1/x)-I^+_{G \setminus \epsilon}(1/x)\\
&=&(-1)^{|\mathcal{E}_+|+1+k+d+1}x^{-d-1}I^+_{-(G+\epsilon)}(x)
  -(-1)^{|\mathcal{E}_+|+k+d+1}x^{-d-1}I^+_{-(G\setminus\epsilon)}(x)\\
&=&(-1)^{|\mathcal{E}_+|+k+d+1}x^{-d-1}
\left(-I^+_{-(G+\epsilon)}(x)-I^+_{-(G\setminus\epsilon)}(x)\right).
\end{eqnarray*}
The bipartite graph $-(G+\epsilon)$ is obtained from $G$ by changing sign except $\epsilon$, so $\epsilon$ is negative edge. From Lemma \ref{lem:sumint}, we have $I^+_{-(G+\epsilon)}(x)=I^+_{-G}(x)-I^+_{-G\setminus\epsilon}(x)$. Since the bipartite graph $-(G\setminus\epsilon)$ is obtained from $G$ by changing all the signs and deleting the edge $\epsilon$, we have $-(G\setminus\epsilon)=-G\setminus\epsilon$. Hence we have
\begin{eqnarray*}
I^+_G(1/x)
&=&(-1)^{|\mathcal{E}_+|+k+d+1}x^{-d-1}
\left(-I^+_{-G}(x)+I^+_{-G\setminus\epsilon}(x)-I^+_{-G\setminus\epsilon}(x)\right)\\
&=&(-1)^{|\mathcal{E}_+|+k+1+d+1}x^{-d-1}I^+_{-G}(x),
\end{eqnarray*}
as desired.
\end{proof}
Theorem \ref{thm:signchange} restores the symmetry of the definition of the signed interior polynomial. That is, replacing negative edges with positive edges, the definition of the interior polynomial does not change essentially. To be precise, let us define another signed interior polynomial as follows.  
\begin{definition}
Let $G=(V,E,\mathcal{E}_+\cup\mathcal{E}_-)$ be a signed bipartite graph. We let
\[
I^{-}_{G}\left(x \right)=\sum_{\mathcal{S} \subseteq \mathcal{E}_{+}}(-1)^{|\mathcal{S}|}I'_{G \setminus \mathcal{S}}(x),
\]
where $G\setminus \mathcal{S}$ is obtained from $G$ by deleting the edges in $\mathcal{S}$ and forgetting sign.
\end{definition}

\begin{cor}\label{cor:symmetric}
For any signed bipartite graph $G=(V,E,\mathcal{E}_+\cup\mathcal{E}_-)$, we have
\[
I^-_{G}(x)=I^+_{-G}(x)=(-1)^{|\mathcal{E}_+|+|\mathcal{E}_-|+|E|+|V|-1}x^{|E|+|V|-1}I^+_G(1/x).
\]
\end{cor}

\section{Flyping and mutation}\label{sec:fm}
Flyping and mutation are operations on a link under which the HOMFLY polynomial does not change. First, we discuss flyping (see Figure \ref{fig:knotflype}), by which the tangle included in the thickened disc $R$ undergoes a translation and a $180^\circ$ rotation. When the link is oriented, two of its four strands meeting $\partial R$ point into $R$, and the other two point out of $R$. This leads to a separation of essentially two cases. We will concentrate on the case depicted in Figure \ref{fig:knotflype}. The other case, in which the four strands would appear parallel in the diagram, leads to a change in the Seifert graph that can be treated using \cite[Theorem 2.11]{kato} and a special case of the mutation operation introduced below.

\begin{figure}[htbp]
\begin{tabular}{ccc}
\begin{minipage}{0.4\hsize}
\centering
\includegraphics[width=4cm]{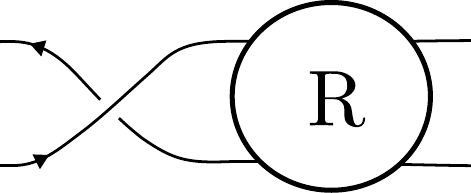}
\end{minipage}
\begin{minipage}{0.2\hsize}
\centering
$\longleftrightarrow$\\
isotopy
\end{minipage}
\begin{minipage}{0.4\hsize}
\centering
\includegraphics[width=4cm]{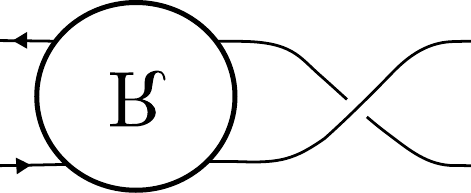}
\end{minipage}
\end{tabular}
\caption{Link flyping}\label{fig:knotflype}
\end{figure}

We examine the Seifert graphs before and after flyping. The tangle inside $R$ yields a subgraph $G_R$ that is connected to the rest at only two vertices. (It may also yield a subgraph or subgraphs connected at only one vertex but those do not lead to any new graph theoretical claims beyond \cite[Corollary 5.4]{KP} and \cite[Theorem 2.11]{kato}.) The left edge moves to the right edge and the color assignments within $G_R$ turn opposite, which we show by writing $\overline{G_R}$ after flyping. (The natural planar embedding of $\overline{G_R}$ would be an upside down version of that of $G_R$, but this does not concern us here.) Signs of edges do not change. This operation is defined for any bipartite graph with a subgraph and edge as in Figure \ref{fig:graphflype} and we call it graph flyping. Under link flyping, the isotopy class, the number of Seifert circles, the number of crossings, and hence the top of the HOMFLY polynomial, are all invariant. We expect that, for any bipartite graph, the interior polynomial does not change under graph flyping.

\begin{figure}[htbp]
\begin{tabular}{ccc}
\begin{minipage}{0.4\hsize}
\centering
\includegraphics[width=4cm]{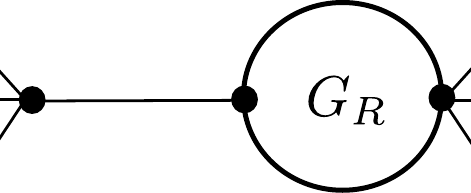}
\end{minipage}\begin{minipage}{0.2\hsize}
\centering
$\longleftrightarrow$\\
\end{minipage}
\begin{minipage}{0.4\hsize}
\centering
\includegraphics[width=4cm]{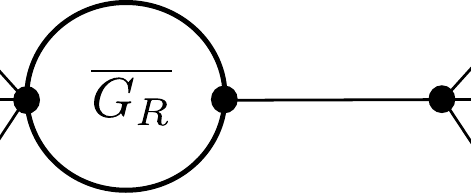}
\end{minipage}
\end{tabular}
\caption{Graph flyping}\label{fig:graphflype}
\end{figure}

Next, we discuss mutation (see Figure \ref{fig:knotmutation}), by which the tangle included in the thickened disc $R$, whose boundary cuts the link at four points, undergoes a $180^\circ$ rotation. We know that the HOMFLY polynomial does not change under mutation \cite[Proposition 2.3]{L}. 

\begin{figure}[htbp]
\centering
\begin{tabular}{cccc}
\begin{minipage}{0.4\hsize}
\hspace{40pt}
\includegraphics[width=3cm]{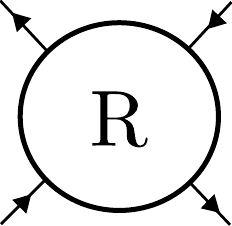}
\end{minipage}
\begin{minipage}{0.2\hsize}
\centering
$\longleftrightarrow$
\end{minipage}
\begin{minipage}{0.4\hsize}
\hspace{10pt}
\includegraphics[width=3cm]{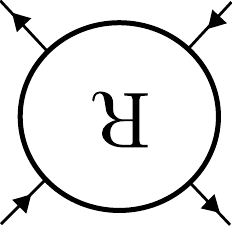}
\end{minipage}
\end{tabular}
\caption{Link mutation}\label{fig:knotmutation}
\end{figure}

In the same way as for flyping, we examine the Seifert graphs before and after mutation. The tangle inside $R$ yields a subgraph $G_R$ which is connected to the rest at only two vertices $v_1$ and $v_2$. (It may happen that $v_1=v_2$ but them the corresponding graph operation is trivial.) The subgraph $G_R$ is rotated with the signs of its edges unchanged. (More precisely, for edges of $G_R$, incidences to $v_1$ become incidences to $v_2$ and) When $v_1$ and $v_2$ are of the same color, then the color classes in $G_R$ do not change. Otherwise color classes in $G_R$ do change, which is denoted by $\overline{G_R}$. This operation is defined for any (bipartite) graph with a subgraph as in Figure \ref{fig:graphmutation} and we call it graph mutation. We expect that, for any bipartite graph, the interior polynomial does not change under graph mutation.

\begin{figure}[htbp]
\begin{tabular}{ccc}
\begin{minipage}{0.4\hsize}
\centering
\includegraphics[width=3cm]{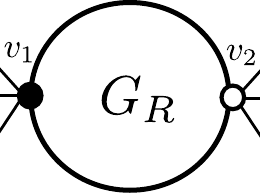}
\end{minipage}\begin{minipage}{0.2\hsize}
\centering
$\longleftrightarrow$\\
\end{minipage}
\begin{minipage}{0.4\hsize}
\centering
\includegraphics[width=3cm]{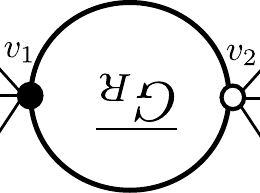}
\end{minipage}\vspace{10pt}
\end{tabular}

\begin{tabular}{ccc}
\begin{minipage}{0.4\hsize}
\centering
\includegraphics[width=3cm]{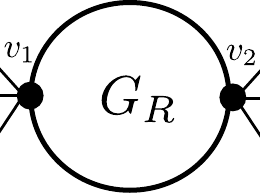}
\end{minipage}\begin{minipage}{0.2\hsize}
\centering
$\longleftrightarrow$\\
\end{minipage}
\begin{minipage}{0.4\hsize}
\centering
\includegraphics[width=3cm]{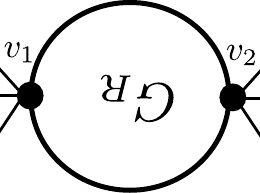}
\end{minipage}
\end{tabular}
\caption{Graph mutation}\label{fig:graphmutation}
\end{figure}

\begin{proof}[Proof of Theorem \ref{thm:nochange}]
First, we will consider the case of flyping in the unsigned case. It is sufficient to prove the claim when the bipartite graph is connected. We use ideas similar to Conway's linear skein theory. Let $G'\star G_R$ and $G'\star \overline{G_R}$ be the bipartite graphs before and after graph flyping.

We take a cycle $\epsilon_1,\delta_1,\epsilon_2,\delta_2,\ldots,\epsilon_n,\delta_n$ of edges in $G_R$. From Theorem \ref{recursion}, we have
\[
I_{G'\star G_R}(x)=\sum_{\emptyset\neq\mathcal{S}\subseteq\{\epsilon_1,\epsilon_2,\ldots,\epsilon_n\}}(-1)^{|\mathcal{S}|-1}I'_{G'\star(G_R\setminus\mathcal{S})}(x).
\]
After graph flyping, we may consider the same cycle in $\overline{G_R}$. For any $\mathcal{S}$, the bipartite graph $G'\star(\overline{G_R}\setminus\mathcal{S})$ is obtained from $G'\star (G_R\setminus\mathcal{S})$ by flyping. We repeat this operation until there are no more cycles in the subgraph. From \cite[Lemma 6.6]{K}, for any bipartite graph $G$, if we construct another bipartite graph $\widetilde{G}$ by adding a new vertex that is connected to just one old vertex, then we have $I_{\widetilde{G}}(x)=I_G(x)$. Now using this repeatedly, we reduce the claim to the cases when $G_R$ is empty or a single path. Flyping does not change these graphs, hence the interior polynomial does not change, either. This completes the proof in the unsigned case.

Next we will treat the case of flyping in signed bipartite graphs. The proof is by induction on the number of negative edges. Let $|\mathcal{E}_-(G)|$ be number of the negative edges in the bipartite graph $G$. When $|\mathcal{E}_-(G'\star G_R)|=0$, the statement holds by the above.

We suppose that the statement holds when the number of the negative edges is less than $k$. When $|\mathcal{E}_-(G'\star G_R)|=k$, we take a negative edge $\epsilon$ in $G'\star G_R$. When $\epsilon$ is in $G'$, by Lemma \ref{lem:sumint}, we have $I^+_{G'\star G_R}(x)=I^+_{(G'+\epsilon)\star G_R}(x)-I^+_{(G'\setminus \epsilon)\star G_R}(x)$. By the inductive hypothesis, we have $I^+_{(G'+\epsilon)\star G_R}(x)=I^+_{(G'+\epsilon)\star \overline{G_R}}(x)$ and $I^+_{(G'\setminus \epsilon)\star G_R}(x)=I^+_{(G'\setminus \epsilon)\star \overline{G_R}}(x)$. Therefore we have
\begin{eqnarray*}
I^+_{G'\star G_R}(x)
&=&I^+_{(G'+\epsilon)\star G_R}(x)-I^+_{(G'\setminus \epsilon)\star G_R}(x)\\ &=&I^+_{(G'+\epsilon)\star \overline{G_R}}(x)-I^+_{(G'\setminus \epsilon)\star \overline{G_R}}(x)\\
&=&I^+_{G'\star \overline{G_R}}.
\end{eqnarray*}
When $\epsilon$ is in $G_R$, we may also think of it as an edge in $\overline{G_R}$. We remark that the bipartite graph $G'\star (\overline{G_R}+\epsilon)$ is obtained from $G'\star (G_R+\epsilon)$ by flyping and that the bipartite graph $G'\star (\overline{G_R}\setminus\epsilon)$ is obtained from $G'\star (G_R\setminus\epsilon)$ by flyping. Since the number of negative edges in these bipartite graphs is $k-1$, by the inductive hypothesis, we have 
\begin{eqnarray*}
I^+_{G'\star G_R}(x)
&=&I^+_{G'\star (G_R+\epsilon)}(x)-I^+_{G'\star (G_R\setminus \epsilon) }(x)\\ &=&I^+_{G'\star \overline{(G_R+\epsilon)}}(x)-I^+_{G'\star \overline{(G_R\setminus \epsilon)}}(x)\\
&=&I^+_{G'\star (\overline{G_R}+\epsilon)}(x)-I^+_{G'\star (\overline{G_R} \setminus \epsilon) }(x)\\
&=&I^+_{G'\star \overline{G_R}}.
\end{eqnarray*}
Therefore the theorem, in the case of flyping, also holds when $|\mathcal{E}_-(G'\star G_R)|=k$ which finishes the proof by induction.

The mutation case follows by the same method as the flyping case, by applying Theorem \ref{recursion} and Lemma \ref{lem:sumint} repeatedly.
\end{proof}

\end{document}